\documentclass{amsart}
\usepackage{amssymb}
\usepackage{graphicx} 
\usepackage[x11names]{xcolor}
\usepackage{pinlabel} 
\usepackage{colortbl} 
\usepackage{enumitem} 
\usepackage{comment} 
\usepackage{appendix}
\usepackage[normalem]{ulem} 

\definecolor{darkblue}{rgb}{0,0,0.4} 
\usepackage{xr-hyper}
\usepackage[colorlinks=true, citecolor=darkblue, filecolor=darkblue, linkcolor=darkblue,urlcolor=darkblue]{hyperref} 
\usepackage[all]{hypcap}
\usepackage{crossreftools}

\newtheorem{theorem}{Theorem}[section]
\newtheorem{lemma}[theorem]{Lemma}
\newtheorem{proposition}[theorem]{Proposition}

\theoremstyle{remark}
\newtheorem{remark}[theorem]{Remark}
\theoremstyle{definition}

\numberwithin{equation}{section}

\DeclareMathOperator{\lk}{lk}
\DeclareMathOperator{\ks}{ks}
\DeclareMathOperator{\Arf}{\mathrm{Arf}}
\DeclareMathOperator{\Int}{\mathrm{Int}}
\newcommand{\del}{\partial}
\newcommand{\N}{\mathbb{N}}
\newcommand{\R}{\mathbb{R}}
\newcommand{\Z}{\mathbb{Z}}
\newcommand{\set}[1]{\left\{#1\right\}}

\renewcommand{\a}{\alpha}
\renewcommand{\b}{\beta}

\newcommand{\h}{\eta}

\newcommand{\CP}[1]{\mathbb{CP}^{#1}}
\newcommand{\bCP}[1]{\overline{\mathbb{CP}^{#1}}}

\title{Smoothly slice links in $\CP2\#\bCP2$}
\author{Marco Marengon and Clayton McDonald}
\date{}

\begin{document}

\begin{abstract}
We show that there exists a link with 2 components which is not smoothly slice in $\CP2\#\bCP2$.
By contrast, it is well-known that every knot (i.e., link with 1 component) is smoothly slice therein.
Our proof uses classical topological and smooth obstructions, as well as constructive arguments to exploit the symmetries of the problem.
As a consequence, we show that there are infinitely many integer homology 3-spheres such that if any of them bounds a ribbon integer homology 4-ball, than there exists an exotic $\CP2\#\bCP2$.
\end{abstract}

\maketitle

\tableofcontents

\section{Introduction}

A popular strategy to disprove the 4-dimensional Poincar\'e conjecture was formulated by Freedman-Gompf-Morrison-Walker in \cite{FGMW}: simply put, their idea is that if a knot $K \subset S^3$ bounds a smooth disc in a homotopy 4-ball $X$, but not in the standard 4-ball $B^4$ (i.e., it is not slice), then $X$ must be exotic.

Sliceness of knots in 4-manifolds has been used to great success in studying exotic manifolds, see \cite{Akbulut} and \cite[Exercise 9.4.23]{GS}, but the particular case of compact manifolds with $S^3$ boundary has remained elusive.

For the purposes of this paper, we say that  $K \subset S^3$ is \emph{slice in $X$} if it bounds a smooth disc in $X^\circ := X \setminus B^4$, which is called a \emph{slice disc}. Classically, this problem has mostly been studied in the 4-sphere (so that $X^\circ = B^4$), but recent progress has been made on the study of slice knots in 4-manifolds other than $S^4$, see for example \cite{KR, MMSW, Ren1, Ren2, MM:slice}. A recent result shows that exotic pairs can be detected by studying \emph{null-homologous} slice discs \cite{MMP}, but the question of whether the set of knots slice in a given 4-manifold can detect exotic structures is still open.

In this paper we consider the analogous question for links.
Here and in what follows we say that an $n$-component link $L$ is \emph{smoothly} or \emph{topologically (strongly) slice in $X$} if $L$ bounds a collection of $n$ disjoint smooth or locally flat discs in $X^\circ$.

%

\begin{theorem}
    \label{thm:CP2bCP2}
    The 2-component link in Figure \ref{fig:CP2bCP2link} is not smoothly slice in $\CP2 \# \bCP2$.
\end{theorem}


\begin{figure}[h]
    \centering
\begingroup%
  \makeatletter%
  \providecommand\color[2][]{%
    \errmessage{(Inkscape) Color is used for the text in Inkscape, but the package 'color.sty' is not loaded}%
    \renewcommand\color[2][]{}%
  }%
  \providecommand\transparent[1]{%
    \errmessage{(Inkscape) Transparency is used (non-zero) for the text in Inkscape, but the package 'transparent.sty' is not loaded}%
    \renewcommand\transparent[1]{}%
  }%
  \providecommand\rotatebox[2]{#2}%
  \newcommand*\fsize{\dimexpr\f@size pt\relax}%
  \newcommand*\lineheight[1]{\fontsize{\fsize}{#1\fsize}\selectfont}%
  \ifx\svgwidth\undefined%
    \setlength{\unitlength}{270.79458402bp}%
    \ifx\svgscale\undefined%
      \relax%
    \else%
      \setlength{\unitlength}{\unitlength * \real{\svgscale}}%
    \fi%
  \else%
    \setlength{\unitlength}{\svgwidth}%
  \fi%
  \global\let\svgwidth\undefined%
  \global\let\svgscale\undefined%
  \makeatother%
  \begin{picture}(1,0.36996397)%
    \lineheight{1}%
    \setlength\tabcolsep{0pt}%
    \put(0,0){\includegraphics[width=\unitlength,page=1]{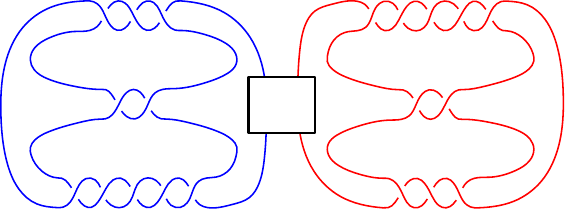}}%
    \put(0.46453481,0.17039761){\color[rgb]{0,0,0}\makebox(0,0)[lt]{\lineheight{1.25}\smash{\begin{tabular}[t]{l}$+29$\end{tabular}}}}%
  \end{picture}%
\endgroup%

    \caption{A 2-component link which is not smoothly slice in $\CP2\#\bCP2$.
    There are $29$ positive full twists in the central box.}
    \label{fig:CP2bCP2link}
\end{figure}

An old result of Norman \cite{N:slice} and Suzuki \cite{S:slice} shows that every knot is slice in $S^2 \times S^2$ and in $\CP2 \#\bCP2$.
On the other hand, in \cite{MY:generalized} Miyazaki-Yasuhara showed that there is a 2-component link that is not topologically slice in $S^2 \times S^2$.
The argument in \cite{MY:generalized} is specific to $S^2 \times S^2$ (see Section \ref{sec:S2xS2} for details), and in particular it left the case of $\CP2 \# \bCP2$ open, which is addressed in Theorem \ref{thm:CP2bCP2} \emph{in the smooth category}.

Our argument to prove Theorem \ref{thm:CP2bCP2} can be adapted to construct 2-components links that are not smoothly slice in $S^2 \times S^2$, and this will be expanded in a separate note. On the other hand, our proof is essentially smooth, so it is unclear if the link in Figure \ref{fig:CP2bCP2link} is topologically slice in $\CP2\#\bCP2$.
While this may seem a drawback of our method, the construction of links that are not \emph{smoothly} slice can be a potential asset in the search of exotic manifolds. 

Recall that $S^2 \times S^2$ and $\CP2\#\bCP2$ are simply connected 4-manifolds that are currently not known to support exotic structures. By contrast, we remark that their (common) blow up $\CP2\#2\bCP2$ admits exotic copies \cite{AP}.
Theorem \ref{thm:CP2bCP2} can theoretically be the starting point for the detection of an exotic $\CP2\#\bCP2$, using a strategy similar to that of \cite{FGMW}.

\begin{proposition}
\label{prop:intro-exotic}
    There exist infinitely many integer homology 3-spheres $\{Y_m\}_{m \in \Z}$ with vanishing Rokhlin invariant such that if any of them bounds an integer homology 4-ball with $\pi_1$ normally generated by the boundary, then there exists an exotic $\CP2\#\bCP2$.
\end{proposition}

These integer homology spheres are constructed as surgeries on links such as in Figure \ref{fig:CP2bCP2link} (or on any link that is not smoothly slice in $\CP2 \# \bCP2$); we show that it is possible to choose such an infinite family with vanishing Rokhlin invariant.
With extra care, one can find similar \textit{rational} homology spheres such that if any of them bounds a $\pi_1$-normal rational ball, then there exists an exotic $\CP2 \# \bCP2$.

This theoretical application is possible since the proof of Theorem \ref{thm:CP2bCP2} uses an essentially smooth (as opposed to topological) ingredient, namely the smooth genus function.
The genus function was successfully employed in \cite{ACMPS} to give a new proof that the (2,1)-cable of the figure-eight knot is not smoothly slice, after a Heegaard Floer theoretic proof was given a year earlier \cite{DKMPS}.

The set of links that are slice in a given 4-manifold contains more precise information than the set of slice knots. In contrast to Norman-Suzuki, a result of Yasuhara implies that for every 4-manifold $X$ there are $(b_2(X)+2)$-component links that are not topologically slice in $X$ \cite{Y:LinkHomology}.
This number of components can be non-minimal. For example, for $\CP2\#\bCP2$ it is 4, but a finer argument shows that there is a 3-component link that is not topologically slice in $\CP2\#\bCP2$ (see Theorem \ref{thm:3-cpt-link-not-top-slice-in-CP2+bCP2}).
It remains an open question whether there is a 2-component link that is not \emph{topologically slice} in $\CP2\#\bCP2$.
In Proposition \ref{prop:b2+1} below, we provide a more efficient (i.e., with fewer components) construction for spin manifolds.

While for every compact 4-manifold there is a link not slice in it, a simple Kirby-diagrammatic argument shows that every link is (smoothly, hence topologically) slice in \emph{some} 4-manifold.
This is a generalisation of the Norman-Suzuki trick \cite{N:slice, S:slice}. (Compare also with \cite[Proposition 4.1]{KR}.)

\begin{proposition}
\label{prop:allslice1}
Let $L \subset S^3$ be an $n$-component link, and let $L_s \subset S^3$ be a sublink of $L$ of $m$ components such that $L_s$ is strongly slice in $B^4$. Then, for all $n_1, n_2 \in \N$ such that $n_1 + n_2 = n-m$, $L$ is smoothly slice in $(\#^{n_1}(S^2 \times S^2)) \# (\#^{n_2}(\CP2 \# \bCP2))$.
\end{proposition}

Our strategy to prove Theorem \ref{thm:CP2bCP2} goes as follows. We start by considering a family of 2-component links which have a certain structure and symmetry (see Figure \ref{fig:S2xS2structure}), and we use a series of obstructive methods:
\begin{itemize}
    \item the Arf invariant;
    \item the Levine-Tristram signature function;
    \item the smooth genus function on $\CP2\#\bCP2$.
\end{itemize}
Each of the above methods is effective at obstructing the existence of slice discs only in given homology classes, so the bulk of our work is to find a way to combine the methods above to eliminate all possible homology classes. In order to do so, we will start making assumptions on the link, and we finally prove that there exists a link satisfying all the assumptions we have made.
For the reader's convenience, we list all the assumptions we made in the Appendix at page \pageref{appendix:CP2bCP2}.

\begin{remark}
As we already mentioned above, our method to prove the existence of 2-component non-slice links works only in the smooth category, since it makes essential use of the smooth genus function on $S^2 \times S^2$ and $\CP2\#\bCP2$. The topological genus function on $S^2 \times S^2$ and $\CP2\#\bCP2$ is not known, but even if it were our method would not apply, because every primitive homology class is represented by a torus \cite{LW:locallyflatspheres, KPTR:surfaces}.
\end{remark}

\subsection{Acknowledgements}
We thank Akira Yasuhara for pointing us towards earlier work on this topic.
We are very grateful to Andr\'as Stipsicz and Marco Golla for their support and helpful discussions, and to Brendan Owens for spotting a mistake in a draft of this paper.
We also thank Roberto Gim\'enez Conejero and Daniele Dona for a helpful conversation.
MM acknowledges that:
This project has received funding from the European Union’s Horizon 2020 research and innovation programme under the Marie Sk{\l}odowska-Curie grant agreement No.\ 893282.

\section{Review of some obstructive methods}

Here we review some methods to obstruct the existence of a properly embedded surface of genus $g$ in a 4-manifold $X$ with boundary a given knot $K \subset S^3$. The methods we review in this section are \emph{topological}, i.e.\ they work for locally flat embeddings in topological 4-manifolds.

We refer to \cite{MMP} for a more detailed description of the state of the art of obstructive methods, and we list only the results that we will need for the scope of this paper.

\subsection{Levine-Tristram signatures}

The following theorem gives an effective obstruction in terms of the Levine-Tristram signatures of a knot when the class of the surface is not primitive.

\begin{theorem}[{\cite{V:G-signature, G:G-signature}, see \cite[Theorem 3.6]{MMP} for this statement}]
\label{thm:signature}
    Let $X$ be a topological closed oriented 4-manifold with $H_1(X;\Z)=0$.
    Let $\Sigma \subset X^\circ$ be a locally flat, properly embedded surface of genus $g$, with boundary a knot $K \subset S^3$.
    If the homology class $[\Sigma] \in H_2(X^\circ, \partial X^\circ;\Z) \cong H_2(X;\Z)$ is divisible by a prime power $m=p^k$, then
    \[
    \left| 
    \sigma_K(e^{2\pi r i/m}) + \sigma(X) - \frac{2r(m-r) \cdot [\Sigma]^2}{m^2}
    \right|
    \leq
    b_2(X) + 2g,
    \]
    for every $r=1, \ldots, m-1$.
\end{theorem}

We recall that the signatures of a satellite of a knot can be computed by the following formula.

\begin{theorem}[{\cite[Theorem 2]{L:signatures}}]
\label{thm:satellite}
    Let $C$ be a knot and $P$ be a pattern with winding number $w$. Then for every root of unity $\zeta$
    \[
    \sigma_{P(C)}(\zeta) = \sigma_{C}(\zeta^w) + \sigma_{P}(\zeta),
    \]
    where $P(C)$ denotes the satellite of $C$ with pattern $P$.
\end{theorem}

\subsection{Arf invariant}

Another obstruction comes from the Arf invariant. In Theorem \ref{thm:Arf} below, $\ks(X)$ denotes the Kirby-Siebenmann invariant of a topological, closed 4-manifolds, and $\Arf(X, \Sigma)$ denotes the Arf invariant of a particular quadratic enhancement of the intersection form on $H_1(\Sigma;\Z/2\Z)$, which comes from viewing $\Sigma$ as sitting in $X$ (see \cite{FK:Rochlin} for details).

The stated form of Theorem \ref{thm:Arf} is the one from \cite[Theorem 3.1]{MMP}. As explained there, the smooth version of it is found in the literature as \cite[p.\ 69, Corollary 6]{K:topology} and \cite[Theorem 2.2]{Y:CL}, and the topological version can be deduced from it using a formula for the Kirby-Siebenmann invariant in the closed case (such a formula can be found e.g.\ in \cite[p.\ 502]{S:wild}). Theorem \ref{thm:Arf} can also be recovered as a special case of \cite[Theorem 4]{Klug}.

For our purposes we need only a simpler version of Theorem \ref{thm:Arf}, because $\ks(X) \equiv 0$ if $X$ admits a smooth structure, and $\Arf(X, \Sigma) \equiv 0$ if $\Sigma$ is a disc.

\begin{theorem}
\label{thm:Arf}
Let $X$ be a topological, closed, connected, oriented 4-manifold.
If $\Sigma \subset X^\circ$ is a properly embedded, locally flat characteristic surface with boundary a knot $K$, then
\[
\frac{\sigma(X) - [\Sigma]^2}8 \equiv \Arf(K) + \Arf(X, \Sigma) + \ks(X) \pmod2
\]
\end{theorem}

\section{Non-slice links in 4-manifolds}

It is well known (cf.\ \cite{N:slice, S:slice}) that every knot is smoothly slice in $S^2 \times S^2$ and in $\CP2 \# \bCP2$. A simple Kirby calculus argument generalises this result.

{\renewcommand{\thetheorem}{\ref{prop:allslice1}}
\begin{proposition}
Let $L \subset S^3$ be an $n$-component link, and let $L_s \subset S^3$ be a sublink of $L$ of $m$ components such that $L_s$ is strongly slice in $B^4$. Then, for all $n_1, n_2 \in \N$ such that $n_1 + n_2 = n-m$, $L$ is smoothly slice in $(\#^{n_1}(S^2 \times S^2)) \# (\#^{n_2}(\CP2 \# \bCP2))$.
\end{proposition}
\addtocounter{theorem}{-1}
}

\begin{proof} 
Define $L_c := L\setminus L_s$. 
Consider the Kirby diagram for a 4-manifold $X$ given by $m(L)$, the mirror of $L$, where $n_1$ components of $L_c$, and all of the components of $L_s$ are $0$-framed and the other $n_2$ components of $L_c$ are $1$-framed, together with a 0-framed meridian added to each link component of $L_c$.
By construction $L$ is strongly slice in $X$. To identify $X$, we unlink $m(L_c)$ by sliding over the 0-framed meridians whenever necessary, and the final Kirby diagram will be for the 0-trace of $L_s$ connect sum the manifold $(\#^{n_1}(S^2 \times S^2)) \# (\#^{n_2}(\CP2 \# \bCP2))$. Because $L_s$ is strongly slice in $B^4$, we can complete the Kirby diagram to a diagram of $(\#^{n_1}(S^2 \times S^2)) \# (\#^{n_2}(\CP2 \# \bCP2))$ using the complement of the 0-trace.
\end{proof}

The above result shows that there exist non-compact 4-manifolds such that every link is slice therein (for example, one can take $\R^4 \#^\infty (S^2 \times S^2)$). However, Yasuhara showed that for every \emph{compact} 4-manifold $X$ there exists a link that is not even null-homologous in $X$, let alone slice. Following Yasuhara, we say that a link $L$ is called \emph{null-homologous} in $X$ if the components of $L$ bound pairwise disjoint surfaces in $X \setminus B^4$. This definition can be given in the smooth or in the topological category; Yasuhara's argument is topological, so the next theorem applies to both cases.

\begin{theorem}[{\cite{Y:LinkHomology}}]
\label{thm:Yasuhara-null-homologous-links}
An $\ell$-component link $L \subset S^3$ is null-homologous in an oriented, compact 4-manifold $X$ if and only if there exist $\ell$ homology classes that pair with each other according to the matrix of linking numbers $\lk(L_i, L_j)$.

In particular, for every oriented, compact 4-manifold $X$, there exists a $(b_2(X)+2)$-component link in $S^3$ that is homologically essential in $X$.
\end{theorem}

In the case of even, simply connected 4-manifolds, we can improve Yasuhara's bound by 1 if we want to obstruct sliceness as opposed to being null-homologous.

\begin{proposition}
\label{prop:b2+1}
    Let $X$ be a closed, even, simply-connected, topological manifold. Then there is an $(b_2(X)+1)$-component link $L$ which is not topologically slice in $X$.
\end{proposition}

\begin{proof}
Let $M$ denote the matrix representing $Q_X$, and construct a $b_2(X)$-component link $L'$ such that every component is an unknot and its pairwise linking numbers are given by the non-diagonal entries of the matrix $M$.
Let $L = L' \sqcup K$, where $K$ is a knot split from $L'$. If $L$ is strongly slice in $X$, then by capping off all the (unknotted) components of $L'$ with a disc in $B^4$, we get a collection of spheres $S_1, \ldots, S_{b_2(X)}$ embedded in $X$ with intersection pattern mod 2 given by the matrix $M$ (note that we know nothing of the self-intersection numbers, except that they are even).
Since $\det M \equiv 1 \pmod 2$, these spheres must be linearly independent in $H_2(X;\Z)$, and therefore they must rationally span all of it. Thus, $K$ must be $H$-slice, which we can obstruct using the signature: we choose $K$ with $\sigma_K(-1) > 2b_2^-(X) = b_2(X) - \sigma(X)$, and Theorem \ref{thm:signature} shows that $K$ cannot be H-slice in $X$.
\end{proof}

\begin{remark}
We remark that Proposition \ref{prop:b2+1} does not hold if you replace sliceness with null-homologousness. In particular, all 3-component links bound pairwise disjoint surfaces in $S^2 \times S^2$ (which is spin) and $\CP2 \# \bCP2$. Using Theorem \ref{thm:Yasuhara-null-homologous-links}, if we denote the triple of linking numbers by $(a,b,c)$, then we just need to find 3 homology classes in the $S^2 \times S^2$ and $\CP2 \# \bCP2$ which pairwise pair to give $a$, $b$, and $c$ respectively. These classes can be chosen to be the columns of the following matrices, where $a'= \frac{a}{(a,b)}$, $b'= \frac{b}{(a,b)}$, and $sb'+ra'=1$:
\[
\left(
\begin{matrix}
    (a,b) & sc & rc \\
    0 & a' & b' \\
\end{matrix}
\right)
\qquad \qquad
\left(
\begin{matrix}
    1 & a & b \\
    0 & 1 & ab-c \\
\end{matrix}
\right)
\]
\end{remark}

\section{A 2-component link not topologically slice in \texorpdfstring{$S^2 \times S^2$}{S2 x S2}}
\label{sec:S2xS2}

Miyazaki and Yasuhara provided an example of a 2-component link that is not topologically slice in $S^2 \times S^2$. We give an overview of their argument here.

\begin{theorem}[{\cite{MY:generalized}}]
    The 2-component link in Figure \ref{fig:Miyazaki-Yasuhara} on the right is not smoothly slice in $S^2 \times S^2$.
\end{theorem}

\begin{figure}
    \centering
    \resizebox{0.7\textwidth}{!}{\includegraphics{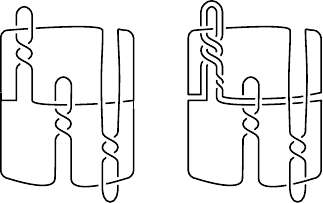}}
    \caption{On the left-hand side there is a $\Theta$ graph where each constituent knot is the figure-eight knot. The right-hand side shows a 2-component link constructed from it which is not topologically slice in $S^2 \times S^2$.}
    \label{fig:Miyazaki-Yasuhara}
\end{figure}

\begin{proof}[{Sketch of the proof from \cite{MY:generalized}}]
As a first step, they show that the figure-eight knot $4_1$ does not bound a disc $\Delta$ in $S^2 \times S^2$ in a \emph{characteristic} homology class: this is because $\Arf(4_1) \equiv 1$ and $\sigma(4_1) = 0$, so one can apply Theorem \ref{thm:Arf} (which says $|[\Delta]^2| \geq 8$) and Theorem \ref{thm:signature} (which says that $|[\Delta]^2| \leq 4$). 

Next, they appeal to a flexible construction of Kinoshita \cite{Kinoshita}, who showed that every triple of knots $(K_1, K_2, K_3)$ can be realised as the three constituent knots of an embedded $\Theta$-graph (i.e., the three knots obtained by gluing two of the three arcs): this construction is in turn based on Suzuki's band presentation of any knot $K$, and illustrated in Figure \ref{fig:Miyazaki-Yasuhara}. Thus, there is an embedded $\Theta$-graph whose constituent knots are all figure-eight knots.
One can then define $L = A \cup B$ as the union of two constituent knots of such a $\Theta$-graph, and, by adding a twist along the common arc if necessary, we can assume that $\lk(A,B)$ is even.
Note that the band sum $A \#_b B$ gives the third constituent knot, hence this is also isotopic to the knot $4_1$.

Finally, suppose that $L$ is slice in $S^2 \times S^2$, with discs $D_A$ and $D_B$; note that there is a boundary connected sum $D_A \natural D_B$ with boundary $A \#_b B$.
With a simple algebraic argument, one can then prove that at least one among $[D_A]$, $[D_B]$, and $[D_A \natural D_B] = [D_A] + [D_B]$ must be characteristic (which is a contradiction).
To show this, write $[D_A]=(a,a')$ and $[D_B]=(b,b')$. The assumption that $\lk(A,B) \equiv 0$ implies that $ab'+ba'\equiv 0$. If both summands are odd, then $a \equiv a' \equiv b \equiv b' \equiv 1$, so $[D_A]+[D_B]$ is characteristic. If instead both summands are even, then $ab'\equiv ba' \equiv 0$, and, assuming that $[D_A]$ and $[D_B]$ are not characteristic, we obtain $a \equiv b \equiv 0$ and $a' \equiv b' \equiv 1$ (or vice versa), and therefore the sum $[D_A]+[D_B]$ is characteristic.
\end{proof}

\begin{remark}
    The argument of Miyazaki-Yasuhara is very specific to $S^2 \times S^2$.
    For example, note that for $(S^2 \times S^2) \# (S^2 \times S^2)$ the argument in the first paragraph of the proof does not work, because $[\Delta^2]= \pm 8$ is not obstructed; the final algebraic consideration would also more complicated, and one would need to conclude that $[D]^2$ is characteristic and $\neq 8$.
    Another case where Miyazaki-Yasuhara's argument is not well suited is $\CP2 \# \bCP2$ (which is the main focus of this paper), because characteristic classes in this case are not even, so the signature obstruction does not apply to them.
\end{remark}

\section{A 2-component link not smoothly slice in \texorpdfstring{$\CP2\#\bCP2$}{a blown-up CP2}}
\label{sec:CP2+bCP2}

Recall that by the Norman-Suzuki trick all knots are slice in $\CP2\#\bCP2$. However, we cannot use the techniques of Miyazaki and Yasuhara, as characteristic classes are not always divisible. Without such a confluence of invariants, we will need to rule out homology classes using a fairly lengthy case analysis.

{
\renewcommand{\thetheorem}{\ref{thm:CP2bCP2}}
\begin{theorem}
    The 2-component link in Figure \ref{fig:CP2bCP2link} is not smoothly slice in $\CP2\#\bCP2$.
\end{theorem}
\addtocounter{theorem}{-1}
}

For the rest of this section, let $X:=\CP2 \# \bCP2$ and $X^\circ := X \setminus \Int(B^4)$.
Let $A$ and $B$ be the two link components, and suppose that $L$ bounds two disjoint smooth discs $D_A, D_B \subset X^\circ$, so that $\del D_A = A$ and $\del D_B = B$, and let $\a := [D_A]$ and $\b := [D_B]$ denote the homology classes of such discs in $H_2(X^\circ, S^3) \cong H_2(X)$.
The idea of the proof is to combine various obstructive methods to rule out all the possible pairs $(\a,\b)$, hence showing that the link cannot be slice.
To do so, we will progressively add assumptions on $L$ until we eventually eliminate all pairs $(\a,\b)$. All the assumptions are collected together in the Appendix at page \pageref{appendix:CP2bCP2}.

\begin{figure}
\begingroup%
  \makeatletter%
  \providecommand\color[2][]{%
    \errmessage{(Inkscape) Color is used for the text in Inkscape, but the package 'color.sty' is not loaded}%
    \renewcommand\color[2][]{}%
  }%
  \providecommand\transparent[1]{%
    \errmessage{(Inkscape) Transparency is used (non-zero) for the text in Inkscape, but the package 'transparent.sty' is not loaded}%
    \renewcommand\transparent[1]{}%
  }%
  \providecommand\rotatebox[2]{#2}%
  \newcommand*\fsize{\dimexpr\f@size pt\relax}%
  \newcommand*\lineheight[1]{\fontsize{\fsize}{#1\fsize}\selectfont}%
  \ifx\svgwidth\undefined%
    \setlength{\unitlength}{175.70183275bp}%
    \ifx\svgscale\undefined%
      \relax%
    \else%
      \setlength{\unitlength}{\unitlength * \real{\svgscale}}%
    \fi%
  \else%
    \setlength{\unitlength}{\svgwidth}%
  \fi%
  \global\let\svgwidth\undefined%
  \global\let\svgscale\undefined%
  \makeatother%
  \begin{picture}(1,0.54625635)%
    \lineheight{1}%
    \setlength\tabcolsep{0pt}%
    \put(0,0){\includegraphics[width=\unitlength,page=1]{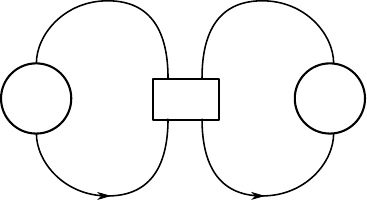}}%
    \put(0.07058651,0.26130829){\color[rgb]{0,0,0}\makebox(0,0)[lt]{\lineheight{1.25}\smash{\begin{tabular}[t]{l}\small $T_A$\end{tabular}}}}%
    \put(0.88102587,0.26130834){\color[rgb]{0,0,0}\makebox(0,0)[lt]{\lineheight{1.25}\smash{\begin{tabular}[t]{l}\small $T_B$\end{tabular}}}}%
    \put(0.49196079,0.26130829){\color[rgb]{0,0,0}\makebox(0,0)[lt]{\lineheight{1.25}\smash{\begin{tabular}[t]{l}\small $n$\end{tabular}}}}%
  \end{picture}%
\endgroup%

    \caption{The structure of the link $L$.}
    \label{fig:S2xS2structure}
\end{figure}
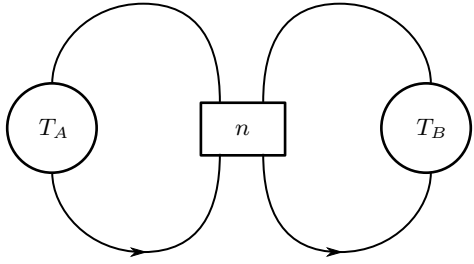

\subsection{The structure of the link \texorpdfstring{$L$}{L}}
We make some assumptions on the structure of $L$ to simplify our case analysis and our computations in the later subsections. Specifically, we assume that:
\begin{enumerate}[font=\textbf]
    \item[({\crtcrossreflabel{A1}[it:A1]})]
    $L$ has a diagram as in Figure \ref{fig:S2xS2structure}, where $T_A$ (resp.\ $T_B$) is a $(1,1)$-tangle whose closure is $A$ (resp.\ $B$), and $n \in \Z$ is the number of right-handed full twists added in the region.
\end{enumerate}

We remark that $n = - \lk(L) = \a \cdot \b$. This follows from the more general statement that if $\Sigma_A, \Sigma_B \subset X^\circ$ are properly embedded surfaces in homology classes $\a$ and $\b$, and with boundary $A$ and $B$ respectively, the following relation holds
\begin{equation}
\label{eq:lkformula}
    \#(\Sigma_A \pitchfork \Sigma_B) + \lk(m(L)) = \a \cdot \b,
\end{equation}
where $m(L)$ denotes the mirror of $L$.
This is because one can cap off $\Sigma_A$ and $\Sigma_B$ in $X$ using Seifert surfaces $F_A$ and $F_B$ for $m(A)$ and $m(B)$ respectively, slightly pushed into $B^4$. Then the intersection number $\a \cdot \b$ is computed by $\#(\Sigma_A \pitchfork \Sigma_B) + \#(F_A \pitchfork F_B)$, and the second summand is well known to agree with $\lk(m(L))$.

Using this form for the link rather than that of Kinoshita exchanges complete control over $A$, $B$, and $A\#_bB$ for more control on $A\#_b B^r$ as well as various cables of $A$ and $B$:

\begin{lemma}
\label{lem:sumsandcablings}
Suppose that $X$ is a smooth, connected 4-manifold, $L = A \cup B$ is a link satisfying \eqref{it:A1}, and that there are two disjoint smooth discs $D_A, D_B \subset X^\circ$, with $\del D_A = A$ and $\del D_B = B$. If $\a := [D_A]$ and $\b := [D_B]$, then:
\begin{itemize}
    \item the knot $A \# B$ bounds a smooth disc in $X^\circ$ in homology class $\a + \b$;
    \item the knot $A \# B^r \# T_{2, 2n\pm1}$ bounds a smooth disc in $X^\circ$ in homology class $\a - \b$;
    \item the knot $A \# (B_{(2,-2\beta^2-2n\pm1)})$ bounds a smooth disc in $X^\circ$ in homology class $\a + 2\b$.
\end{itemize}
Note that $n=-\lk(A,B)$.
\end{lemma}
In the statement above $K^r$ denotes the reverse of the knot $K$ (i.e.\ $K$ with reversed orientation), not to be confused with the mirror $m(K)$ of $K$, and the knot $K_{(p,q)}$ denotes the $(p,q)$-cable of $K$.

\begin{figure}
    \centering
\begingroup%
  \makeatletter%
  \providecommand\color[2][]{%
    \errmessage{(Inkscape) Color is used for the text in Inkscape, but the package 'color.sty' is not loaded}%
    \renewcommand\color[2][]{}%
  }%
  \providecommand\transparent[1]{%
    \errmessage{(Inkscape) Transparency is used (non-zero) for the text in Inkscape, but the package 'transparent.sty' is not loaded}%
    \renewcommand\transparent[1]{}%
  }%
  \providecommand\rotatebox[2]{#2}%
  \newcommand*\fsize{\dimexpr\f@size pt\relax}%
  \newcommand*\lineheight[1]{\fontsize{\fsize}{#1\fsize}\selectfont}%
  \ifx\svgwidth\undefined%
    \setlength{\unitlength}{145.00660573bp}%
    \ifx\svgscale\undefined%
      \relax%
    \else%
      \setlength{\unitlength}{\unitlength * \real{\svgscale}}%
    \fi%
  \else%
    \setlength{\unitlength}{\svgwidth}%
  \fi%
  \global\let\svgwidth\undefined%
  \global\let\svgscale\undefined%
  \makeatother%
  \begin{picture}(1,0.54120713)%
    \lineheight{1}%
    \setlength\tabcolsep{0pt}%
    \put(0,0){\includegraphics[width=\unitlength,page=1]{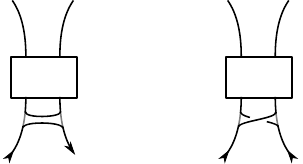}}%
    \put(0.1261099,0.26864146){\color[rgb]{0,0,0}\makebox(0,0)[lt]{\lineheight{1.25}\smash{\begin{tabular}[t]{l}\small $n$\end{tabular}}}}%
    \put(0.83907325,0.26864146){\color[rgb]{0,0,0}\makebox(0,0)[lt]{\lineheight{1.25}\smash{\begin{tabular}[t]{l}\small $n$\end{tabular}}}}%
  \end{picture}%
\endgroup%

    \caption{The figure shows the band surgeries to get $A\#B$ and $A\#B^r \# T_{2,2n\pm1}$, respectively. The sign of the crossing in the band on the right determines whether the torus knot component is $T_{2, 2n-1}$ or $T_{2,2n+1}$.}
    \label{fig:band1}
\end{figure}

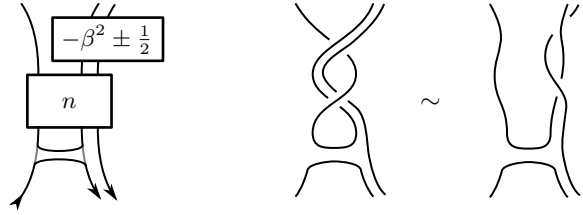
\begin{figure}
    \centering
\begingroup%
  \makeatletter%
  \providecommand\color[2][]{%
    \errmessage{(Inkscape) Color is used for the text in Inkscape, but the package 'color.sty' is not loaded}%
    \renewcommand\color[2][]{}%
  }%
  \providecommand\transparent[1]{%
    \errmessage{(Inkscape) Transparency is used (non-zero) for the text in Inkscape, but the package 'transparent.sty' is not loaded}%
    \renewcommand\transparent[1]{}%
  }%
  \providecommand\rotatebox[2]{#2}%
  \newcommand*\fsize{\dimexpr\f@size pt\relax}%
  \newcommand*\lineheight[1]{\fontsize{\fsize}{#1\fsize}\selectfont}%
  \ifx\svgwidth\undefined%
    \setlength{\unitlength}{216.72384624bp}%
    \ifx\svgscale\undefined%
      \relax%
    \else%
      \setlength{\unitlength}{\unitlength * \real{\svgscale}}%
    \fi%
  \else%
    \setlength{\unitlength}{\svgwidth}%
  \fi%
  \global\let\svgwidth\undefined%
  \global\let\svgscale\undefined%
  \makeatother%
  \begin{picture}(1,0.36211339)%
    \lineheight{1}%
    \setlength\tabcolsep{0pt}%
    \put(0,0){\includegraphics[width=\unitlength,page=1]{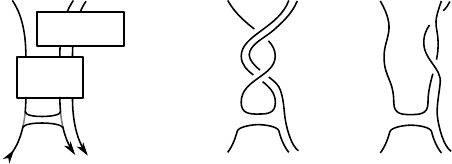}}%
    \put(0.09822061,0.17974381){\color[rgb]{0,0,0}\makebox(0,0)[lt]{\lineheight{1.25}\smash{\begin{tabular}[t]{l}\small $n$\end{tabular}}}}%
    \put(0.09763886,0.28767627){\color[rgb]{0,0,0}\makebox(0,0)[lt]{\lineheight{1.25}\smash{\begin{tabular}[t]{l}\small $-\beta^2 \pm\frac12$\end{tabular}}}}%
    \put(0.72113098,0.17974381){\color[rgb]{0,0,0}\makebox(0,0)[lt]{\lineheight{1.25}\smash{\begin{tabular}[t]{l}$\sim$\end{tabular}}}}%
  \end{picture}%
\endgroup%

    \caption{The figure on the left shows the band sum of $A$ and the cable $B_{2,-2\b^2\pm1}$. 
    (The two parallel copies of $B$ do not wind on each other in the box labelled with $n$.)
    The figure on the right shows that a positive full twist in the box labelled $n$ on the left can be traded for a negative full twist on the cable of $B$.}
    \label{fig:band2}
\end{figure}

\begin{proof}
For the first bullet point, by performing a surgery along a horizontal band below the $n$-labelled box in Figure \ref{fig:S2xS2structure}, then the box can be undone and one obtains $A\#B$.
See the left hand side of Figure \ref{fig:band1}.

If instead one first reverses the orientation of $B$ and then performs a surgery along a band parallel to the $n$-labelled box, but half-twisted, one gets a new knot which is the connected sum of three summands: $A$ appears on the left, $B$ on the right, and $T_{2, 2n\pm1}$ in the middle. (The sign in $\pm1$ depends on the direction of twisting of the band.)
See the right hand side of Figure \ref{fig:band1}.

We now turn to the last bullet point. First, take a parallel copy $D_B'$ of $D_B$, so that the two are disjoint. Then, $\del(D_B \cup D_B') = B_{2,2p}$, an appropriate cable of $B$. Equation \eqref{eq:lkformula} applied to $B_{2,2p}$ implies that $p=-\b^2$. By attaching a half-twisted band to join the two components, we see that $B_{2,-2\b^2\pm1}$ bounds a disc in homology class $2\b$ and supported in a neighbourhood of $D_B$.
Now we consider the original link $L$ and we cable the $B$ component to obtain $B_{2,-2\b^2\pm1}$.
After band summing the components $A$ and $B_{2,-2\b^2\pm1}$ as shown in Figure \ref{fig:band2} on the left, we obtain a knot that bounds a disc in homology class $\a +2\b$. After an isotopy (see Figure \ref{fig:band2} on the right), this knot is identified as $A \# (B_{(2,-2\beta^2-2n\pm1)})$.
\end{proof}



%

\subsection{The genus function on \texorpdfstring{$\CP2\#\bCP2$}{the blown-up CP2}.}

Given a smooth, connected 4-manifold $X$, the 4-ball genus of a knot $K$ gives a first obstruction to the homology classes of $H_2(X^\circ, S^3)$ that are represented by a disc with boundary $K$. 
More precisely, if there is such a disc in homology class $\alpha$, by gluing it to a minimal surface for $K$ in $B^4$ we obtain a closed surface of genus $g_{B^4}(K)$ sitting in a homology class that by abuse of notation we still call $\alpha \in H_2(X) \cong H_2(X^\circ, S^3)$.
Then, knowledge of the genus function on $H_2(X)$ can give obstruction to such an $\alpha$.

Luckily for us, the smooth genus function on $\CP2\#\bCP2$ was determined by Ruberman.

\begin{theorem}[{\cite[Corollary 1.3]{R:minimalgenus}}]
\label{thm:Ruberman2}
    The minimal genus of a smoothly embedded orientable surface in $\CP2 \# \bCP2$ in homology class $(a_1, a_2) \in H_2(\CP2 \# \bCP2) \cong \Z^2$, with respect to the obvious basis, is
    \[
    G_{\CP2 \# \bCP2}(a_1,a_2) =
    \begin{cases}
        \displaystyle\frac{(|a_1|-1)(|a_1|-2)}2 - \frac{|a_2|(|a_2|-1)}2 & \mbox{if $|a_1| > |a_2|$} \vspace{10pt}\\
        0 & \mbox{if $|a_1| = |a_2|$} \vspace{10pt}\\
        \displaystyle\frac{(|a_2|-1)(|a_2|-2)}2 - \frac{|a_1|(|a_1|-1)}2 & \mbox{if $|a_1| < |a_2|$}
    \end{cases}
    \]
\end{theorem}

Let us now return to the link $L = A \cup B$ whose sliceness in $\CP2 \# \bCP2$ we are trying to obstruct. If either $A$ or $B$ were slice in the 4-ball, then by Proposition \ref{prop:allslice1} the link would be slice in $\CP2 \# \bCP2$. The next simplest assumption we can make is the following:
\begin{enumerate}[font=\textbf]
    \item[({\crtcrossreflabel{A2}[it:A2]})]
    $g_{B^4}(A) = g_{B^4}(B) = 1$.
\end{enumerate}

\begin{lemma}
\label{lem:A}
Suppose that $A$ is a knot with $g_{B^4}(A)=1$ which is smoothly slice in $\CP2 \# \bCP2$, with a disc in homology class $\a$. Then $\a = (a_1, a_2)$, where: 
\begin{itemize}
    \item $||a_1|-|a_2||\leq1$; or
    \item $\set{|a_1|,|a_2|} = \set{0,2}$; or
    \item $\set{|a_1|,|a_2|} = \set{0,3}$; or
    \item $\set{|a_1|,|a_2|} = \set{1,3}$.
\end{itemize}
\end{lemma}

\begin{proof}
If $A$ bounds a smooth disc $D_A \subseteq (S^2 \times S^2)^\circ$ in homology class $\a$, then $\a$ is represented by a closed surface of genus 1 (obtained by capping off a minimal genus surface for $A$ in $B^4$ with the slice disc $D_A$), and therefore $G_{\CP2 \# \bCP2}(\a) \leq 1$.

An analysis of the function $G_{\CP2 \# \bCP2}(a_1,a_2)$ shows that the values $0$ and $1$ are attained if and only if $||a_1|-|a_2||\leq1$, $\set{|a_1|,|a_2|} = \set{0,2}$, $\set{|a_1|,|a_2|} = \set{0,3}$, or $\set{|a_1|,|a_2|} = \set{1,3}$.
\end{proof}

\begin{remark}
    This is the only point in our argument where the adjective \emph{smooth} actually makes a difference. As far as the authors know, the topological genus function of $\CP2 \# \bCP2$ (defined using locally flat embeddings as opposed to smooth ones) is not known.
    However, as every primitive class of $\CP2 \# \bCP2$ is represented by a locally flat embedded sphere (by a result of Lee-Wilczy\'nski), there would be many more classes to consider in the case analysis. The paper of Lee-Wilczy\'nski, which was brought to our attention by Arunima Ray, more generally gives an upper bound to the topological genus function by determining exactly when there is a \emph{simple} topological embedding, cf.\ \cite[Theorem 1.2]{LW:locallyflatspheres}.
    See also \cite{KPTR:surfaces}.
\end{remark}

\subsection{Symmetries}

We will use symmetries to reduce the number of pairs of homology classes $(\a,\b)$ that we need to study. In addition to using the symmetries of $X:=\CP2 \# \bCP2$, we will make the following assumption:
\begin{enumerate}[font=\textbf]
    \item[({\crtcrossreflabel{A3}[it:A3]})] The link $L$ in Figure \ref{fig:S2xS2structure} has an ambient isotopy that swaps $A$ and $B$.
\end{enumerate}

With the above assumption, we list all the orientation-preserving symmetries we can use and their action:
\begin{enumerate}[font=\textbf]
    \item[({\crtcrossreflabel{S1}[it:S1]})] complex conjugation on the $\CP2$ summand:
    \newline acts on $H_2(X)$ by $(a_1,a_2)\mapsto(-a_1,a_2)$;
    \item[({\crtcrossreflabel{S2}[it:S2]})] complex conjugation on the $\bCP2$ summand:
    \newline acts on $H_2(X)$ by $(a_1,a_2)\mapsto(a_1,-a_2)$;
    \vspace{2pt}
    \item[({\crtcrossreflabel{S3}[it:S3]})] assumption \eqref{it:A3}:
    \newline acts on pairs $(\a,\b) \in H_2(X) \times H_2(X)$ by $(\a,\b) \mapsto (\b,\a)$.
\end{enumerate}

\subsection{Linking number and Arf invariant}
We first rule out the possibility that $\a$ is of the form $(a, \pm a)$ (or that $\b$ is of the form $(b,\pm b)$). To this end, we make two assumptions:
\begin{enumerate}[font=\textbf]
    \item[({\crtcrossreflabel{A4a}[it:A4a]})] $\lk(A,B) \not\equiv 0 \pmod 2$.
    \item[({\crtcrossreflabel{A5}[it:A5]})] $\Arf A = \Arf B = 1$.
\end{enumerate}
We remark that \eqref{it:A4a} is one of the assumptions that we will make on the linking number. This assumption, together with the upcoming ones \eqref{it:A4b}-\eqref{it:A4f}, will be subsumed in assumption \eqref{it:A4}, which is the one listed in the Appendix at page \pageref{appendix:CP2bCP2}.

\begin{lemma}
\label{lem:aa}
Suppose that $L=A\cup B$ is a link satisfying \eqref{it:A4a} and \eqref{it:A5} which is smoothly slice in $\CP2 \# \bCP2$, with discs in homology classes $\a$ and $\b$. Then $\a$ (and likewise $\beta$) is not of the form $(a,\pm a)$ or $(0,\pm2)$ or $(\pm2,0)$.
\end{lemma}

\begin{proof}
If either $\a$ or $\b$ were of the form $(0,\pm2)$ or $(\pm2,0)$, then $\lk(A,B) = -(a_1 b_1 - a_2 b_2)$ would be even, contradicting \eqref{it:A4}.

If instead $\a=(a,\pm a)$, let $\b = (b_1, b_2)$. Then, by Equation \eqref{eq:lkformula}
\[
a (b_1 \mp b_2) = - \lk(A,B),
\]
which is odd by \eqref{it:A4}. This implies that $a$ is odd too, and that therefore $\a = (a,\pm a)$ is a characteristic class. Then, using Theorem \ref{thm:Arf} we can compute
\[
\Arf A \equiv \frac{\sigma(\CP2 \# \bCP2) - \a^2}8 \equiv 0 \pmod2,
\]
which contradicts assumption \eqref{it:A5}.
\end{proof}

We add a further assumption on the linking number to rule out the classes $(\pm3,0)$ or $(0,\pm3)$:
\begin{enumerate}[font=\textbf]
    \item[({\crtcrossreflabel{A4b}[it:A4b]})] $\lk(A,B) \not\equiv 0 \pmod 3$.
\end{enumerate}

\begin{lemma}
\label{lem:aa2}
Suppose that $L=A\cup B$ is a link satisfying \eqref{it:A4b} which is smoothly slice in $\CP2 \# \bCP2$, with discs in homology classes $\a$ and $\b$. Then $\a$ (and likewise $\b$) is not of the form $(\pm3,0)$ or $(0,\pm3)$.
\end{lemma}
\begin{proof}
If $\a$ were of the form $(\pm3,0)$ or $(0,\pm3)$, then by Equation \eqref{eq:lkformula} the linking number $\lk(A,B)$ would be divisible by 3, contradicting \eqref{it:A4b}.
\end{proof}

By putting together everything obtained so far, we can prove the following lemma.

\begin{lemma}
\label{lem:B1-5}
Suppose that $L=A\cup B$ is a link satisfying \eqref{it:A1}-\eqref{it:A3}, \eqref{it:A4a}, \eqref{it:A4b}, and \eqref{it:A5} which is smoothly slice in $\CP2 \# \bCP2$, with discs in homology classes $\a$ and $\b$.
Then we can assume that \emph{either} both $\a$ and $\b$ belong to an infinite family (not necessarily the same) among the following:
\begin{enumerate}
    \item $(a,a+1)$,
    \item $(a+1,a)$,
    \item $(a,-(a+1))$,
    \item $(a+1,-a)$;
\end{enumerate}
\emph{or} $\a$ belongs to one of the above infinite families while $\b$ is one of the eight sporadic classes cases $(\pm1, \pm3)$ and $(\pm3, \pm1)$.
\end{lemma}

\begin{proof}
Lemmas \ref{lem:A}, \ref{lem:aa}, and \ref{lem:aa2} imply that each of $\a$ and $\b$ belongs to one of the four infinite families or is one of the eight sporadic classes $(\pm1, \pm3)$ and $(\pm3, \pm1)$.

To rule out the possibility that both $\a$ and $\b$ are sporadic classes, we notice that in such a case the linking number $\lk(A,B) = -\a\cdot\b$ would be even, in contradiction with assumption \eqref{it:A4a}.
Thus, either each of $\a$ and $\b$ belongs to one of the infinite families, or exactly one of them is $(\pm1, \pm3)$ or $(\pm3, \pm1)$, in which case by assumption \eqref{it:A3} we can assume it is $\b$.
\end{proof}

\subsection{Ruling out the pairs of infinite families}
Moving from Lemma \ref{lem:B1-5}, we start by obstructing the case when both $\a$ and $\b$ are in one of the infinite families. The possible pairs of homology classes are summarised in Table \ref{tab:lkCP2}.

Noting that symmetries \eqref{it:S1} and \eqref{it:S2} combined permute the 4 infinite families from Lemma \ref{lem:B1-5} transitively, we can restrict to the first row of the table, where $\a$ is of the form $(a, a+1)$.

\begin{table}[]
    \resizebox{\textwidth}{!}{
    \begin{tabular}{c|c|c|c|c}
     & $(b,b+1)$ & $(b+1,b)$ & $(b,-(b+1))$ & $(b+1,-b)$ \\
    \hline
    $(a,a+1)$ & $-(a+b+1)$\cellcolor{Snow2} & $a-b$\cellcolor{LightBlue1} & $2ab + a + b + 1$\cellcolor{PaleGreen1} & $2ab + a + b$\cellcolor{PaleGreen1} \\
    \hline
    $(a+1,a)$ & $b-a$\cellcolor{Tan1} & $a+b+1$\cellcolor{Tan1} & $2ab + a + b$\cellcolor{Tan1} & $2ab + a + b + 1$\cellcolor{Tan1} \\
    \hline
    $(a,-(a+1))$ & $2ab + a + b + 1$\cellcolor{Tan1} & $2ab + a + b$\cellcolor{Tan1} & $-(a+b+1)$\cellcolor{Tan1} & $a-b$\cellcolor{Tan1} \\
    \hline
    $(a+1,-a)$ & $2ab + a + b$\cellcolor{Tan1} & $2ab + a + b + 1$\cellcolor{Tan1} & $b-a$\cellcolor{Tan1} & $a+b+1$\cellcolor{Tan1}
    \end{tabular}
    }
    \vspace{10pt}
    \caption{The table shows the possible pairs of homology classes $(\a,\b)$, for $a,b \in \Z$, under the assumption that both $\a$ and $\b$ are in one of the infinite families from Lemma \ref{lem:B1-5} (not necessarily the same).
    The value in each cell is the intersection number $\a\cdot\b$.
    Cell (1,1) is grey, cell (1,2) is blue, cells (1,3) and (1,4) are green, and the rest are orange. The orange cells can be discarded by symmetry considerations.
    }
    \label{tab:lkCP2}
\end{table}

\subsubsection{Ruling out the grey cell}
We start by ruling out the grey cell of Table \ref{tab:lkCP2}.
We make the following further assumption on the linking number:
\begin{enumerate}[font=\textbf]
    \item[({\crtcrossreflabel{A4c}[it:A4c]})] $|\lk(A,B)| \geq 4$.
\end{enumerate}

\begin{lemma}
\label{lem:table1}
    Suppose that $L=A\cup B$ is a link satisfying \eqref{it:A1}, \eqref{it:A2}, \eqref{it:A4a}, and \eqref{it:A4c} which is smoothly slice in $\CP2 \# \bCP2$, with discs in homology classes $\a$ and $\b$.
    Then the pair $(\a,\b) \neq ((a,a+1),(b,b+1))$ for some $a,b \in \Z$.
\end{lemma}

\begin{proof}
By contradiction assume $(\a,\b) = ((a,a+1),(b,b+1))$. By Equation \eqref{eq:lkformula}, $\lk(A,B) = a+b+1$, and assumption \eqref{it:A4a} then implies that the integer $a+b$ is even.

Therefore, the class $\a+\b = (a+b,a+b+2)$ is 2-divisible, and (using Lemma \ref{lem:sumsandcablings}) we can apply Theorem \ref{thm:signature} to $K=A\#B$, $[\Sigma]=\a+\b$, $m=2$, and $r=1$, which after some manipulation yields
\[
\left|
\sigma_{A\#B}(-1) + 2 \cdot \lk(A,B)
\right|
\leq 2.
\]

By assumption \eqref{it:A2}, the 4-genera of $A$ and $B$ are 1, and therefore their signatures are in absolute value less than $2$, implying $|\sigma_{A\#B}(-1)| \leq 4$. Thus, by the triangle inequality, we get
\[
|2 \cdot \lk(A,B)| \leq |\sigma_{A\#B}(-1)| + |\sigma_{A\#B}(-1) + 2 \cdot \lk(A,B)| \leq 4 + 2,
\]
which contradicts \eqref{it:A4c}.
\end{proof}

\subsubsection{Ruling out the blue cell}

To rule out the blue cell, we make our first assumption on the signature function.
\begin{enumerate}[font=\textbf]
    \item[({\crtcrossreflabel{A6}[it:A6]})] $\sigma_A(\zeta_2) = \sigma_B(\zeta_2) \neq 0$.
\end{enumerate}

\begin{lemma}
\label{lem:table3}
    Suppose that $L=A\cup B$ is a link satisfying \eqref{it:A1}, \eqref{it:A2}, \eqref{it:A4a} and \eqref{it:A6} which is smoothly slice in $\CP2 \# \bCP2$, with discs in homology classes $\a$ and $\b$.
    Then the pair $(\a,\b)$ does not correspond to the blue cell in Table \ref{tab:lkCP2}.
\end{lemma}

\begin{proof}
By contradiction assume $(\a,\b) = ((a,a+1),(b+1,b))$. By Equation \eqref{eq:lkformula}, $\lk(A,B) = b-a$, which is odd by assumption \eqref{it:A4a}.

Therefore, the class $\a+\b = (a+b+1,a+b+1)$ is 2-divisible and with square $0$. Thus, by applying Theorem \ref{thm:signature} as in Lemma \ref{lem:table1} we get
\[
\left|
\sigma_{A\#B}(-1)
\right|
\leq 2,
\]
which contradicts assumption \eqref{it:A6}.
\end{proof}

\subsubsection{Ruling out the green cells}

To rule out the green cells of Table \ref{tab:lkCP2} we make another assumption on the linking number.
\begin{enumerate}[font=\textbf]
    \item[({\crtcrossreflabel{A4d}[it:A4d]})] $\lk(A,B) \notin \set{\pm1, \pm3, \pm5, \pm7, \pm11}$.
\end{enumerate}

\begin{lemma}
\label{lem:table2}
    Suppose that $L=A\cup B$ is a link satisfying \eqref{it:A1}, \eqref{it:A2}, \eqref{it:A4a}, \eqref{it:A4d}, and \eqref{it:A6} which is smoothly slice in $\CP2 \# \bCP2$, with discs in homology classes $\a$ and $\b$.
    Then the pair $(\a,\b)$ does not correspond to one of the green cells in Table \ref{tab:lkCP2}.
\end{lemma}

\begin{proof}
We can treat both cases simultaneously by writing
\[
\a=(a, a+1) \qquad \mbox{and} \qquad \b=(b+\h, -(b+1-\h))
\]
for some $\h \in \set{0,1}$. We can compute the pairing
\begin{equation}
\label{eq:table-a.b}
\begin{aligned}
    \a\cdot\b &= a(b+\h) + (a+1)(b+1-\h) \\
    &= 2ab + a + b - \h + 1,
\end{aligned}
\end{equation}
which is odd by assumption \eqref{it:A4a}, and thus deduce that the integer $a+b-\h$ must be even. It then follows that the class
\[
\a+\b = (a+b+\h, a-b+\h)
\]
is 2-divisible. Using Lemma \ref{lem:sumsandcablings} we can apply Theorem \ref{thm:signature} to $K=A\#B$, $[\Sigma]=\a+\b$, $m=2$, and $r=1$, which after some manipulation yields
\[
\left|
\sigma_{A\#B}(-1) - 2(a+\h)b
\right|
\leq 2.
\]

Arguing as in the proof of Lemma \ref{lem:table1}, we deduce that $|(a+\h)\cdot b| \leq 3$, but assumption \eqref{it:A6} allows us to discard the case $(a+\h)\cdot b=0$, as in Lemma \ref{lem:table3}. Thus, we get
\begin{equation}
\label{eq:table2-1}
(a+\h)\cdot b \in \set{\pm1, \pm2, \pm3}.
\end{equation}

For each of the possible pairs $(a+\h,b)$ satisfying Equation \eqref{eq:table2-1}, we will compute the pairing $\a\cdot\b$, which equals $-\lk(A,B)$, and check that it is one of the values listed in assumption \eqref{it:A4d}.

If we let $\tilde{a} = a+\h$, then the possible values of the pairs $(\tilde{a}, b)$ satisfying Equation \eqref{eq:table2-1} are
\begin{equation}
\label{eq:12cases}
\begin{aligned}
    &(1,1) & \quad & (1,-1) & \quad & (-1,1) & \quad & (-1,-1) \\
    &(1,3) & \quad & (1,-3) & \quad & (-1,3) & \quad & (-1,-3)
\end{aligned}
\end{equation}
and the ones obtained from these by swapping the two coordinates.
Note that we do not need to consider the case when $\tilde{a} b = \pm2$, because this has no solutions under our assumption (noted earlier) that $\tilde{a} + b = a + b + \h$ is an even number.

To simplify the computation of $\a\cdot\b$ for all these cases, we first note that the expression \eqref{eq:table-a.b} is symmetric in $a$ and $b$, hence we can re-write
\begin{equation}
\label{eq:table-a.b2}
\begin{aligned}
    \a\cdot\b &= (a+\h)b + (a+1-\h)(b+1) \\
    &= \tilde{a}b + (\tilde{a}+1-2\h)(b+1).
\end{aligned}
\end{equation}

To treat the two cases of $\h$ simultaneously, we denote $f_{\h}(\tilde{a}, b)$ as the function from Equation \eqref{eq:table-a.b2}. It is straightforward to check that
\[
    f_{1}(\tilde{a}, b) = -f_{0}(-\tilde{a}, b),
\]
so it is enough to compute the values of $f_{\h}(\tilde{a}, b)$ in the case $\h=0$ (and remember to allow for a potential sign change).
A further simplification comes from the fact that
\[
f_{0}(\tilde{a}, b) = f_{0}(b, \tilde{a}),
\]
so it is enough to compute $f_{0}$ in the 8 cases from Equation \eqref{eq:12cases}. These values are straightforward to compute, and they are
\[
\begin{aligned}
    &5 & \quad & -1 & \quad & -1 & \quad & 1 \\
    &11 & \quad & -7 & \quad & -3 & \quad & 3
\end{aligned}
\]
These integers and their opposites are exactly the numbers listed in the set of assumption \eqref{it:A4d}. It follows that if assumption \eqref{it:A4d} holds the link $L=A \cup B$ cannot be smoothly slice in $\CP2 \# \bCP2$ with discs in homology classes $\a$ and $\b$ corresponding to one of the green cells in Table \ref{tab:lkCP2}.
\end{proof}

\subsection{Ruling out the sporadic cases}

Referring to the statement of Lemma \ref{lem:B1-5}, we still have to rule out the case when $\a$ belongs to one of the infinite families and $\b$ is one of the eight sporadic cases. We first use the symmetries to reduce the cases we have to deal with.

\begin{lemma}
\label{lem:reduction}
Suppose that $L=A\cup B$ is a link satisfying \eqref{it:A1}, \eqref{it:A2}, and \eqref{it:A4d} which is smoothly slice in $\CP2 \# \bCP2$, with discs in homology classes $\a$ and $\b$.
Suppose further that $\a$ belongs to one of the infinite families
\begin{enumerate}
    \item $(a,a+1)$,
    \item $(a+1,a)$,
    \item $(a,-(a+1))$,
    \item $(a+1,-a)$;
\end{enumerate}
while $\b$ is one of the eight sporadic classes cases $(\pm1, \pm3)$ and $(\pm3, \pm1)$. Then up to symmetries \eqref{it:S1} and \eqref{it:S2} we can assume that $\a = (a,a+1)$ for some $a\in\Z$ and $\b = (3,1)$ or $(-1,-3)$.
\end{lemma}

\begin{proof}
Since the symmetries \eqref{it:S1} and \eqref{it:S2} combined permute the 4 infinite families transitively, we can restrict to the case of $\a = (a, a+1)$.

If we let $\b=(b_1,b_2)$, we can study 4 cases depending on the value of $b_2-b_1$, which can be $\pm2$ or $\pm4$.

\underline{Case 1: $b_2-b_1=4$.}
In such a case the class $\a+\b$ is of the form $(x,x+5)$, and by Lemma \ref{lem:sumsandcablings} it is represented by a closed surface of genus 2, obtained by capping off a minimal 4-genus surface for $m(A \# B)$ in $B^4$ with a slice disc in $\CP2 \# \bCP2$.
Thus, the genus function (see Theorem \ref{thm:Ruberman2}) implies that $x=-2$ or $x=-3$.
Assume that $\b=(-1,3)$: then for each possible value of $x$, we can compute $\a$ and $\a\cdot\b$, which is done in Table \ref{tab:+4}.
All the computed values of $\a\cdot\b$ are obstructed by assumption \eqref{it:A4d}, so we can rule out the case $\b=(-1,3)$.
As for the case when $\b=(-3,1)$, we observe that the transformation $(x_1,x_2) \mapsto (-x_2,-x_1)$ preserves the infinite family $(a,a+1)$ and swaps $(-3,1)$ with $(-1,3)$: therefore, it preserves the set of possible values $\a\cdot\b$ up to multiplication by $-1$. These values are still obstructed by assumption \eqref{it:A4d}.

\begin{table}[]
    \begin{tabular}{c||c|c}
     $(x,x+5)$ & $(-2,3)$ & $(-3,2)$ \\
    \hline
     $\a$ & $(-1,0)$ & $(-2,-1)$ \\
    \hline
     $\a\cdot\b$ & $1$ & $5$
    \end{tabular}
    \vspace{10pt}
    \caption{Computations of $\a$ and $\a\cdot\b$ under the assumption that $\b=(-1,3)$ and $\a+\b$ is of the form $(x,x+5)$.}
    \label{tab:+4}
\end{table}

\underline{Case 2: $b_2-b_1=-4$.}
We argue in the same way as Case 1.
In this second case, the class $\a+\b$ is of the form $(x,x-3)$, and as before it is represented by a closed surface of genus 2. Thus, the genus function implies that $x\in\set{0,1,2,3}$.
Assuming that $\b=(1,-3)$, the computations of $\a$ and $\a\cdot\b$ for each possible value of $x$ are summarised in Table \ref{tab:-4}.
All the computed values of $\a\cdot\b$ are obstructed by assumption \eqref{it:A4d}. The other possible case, namely $\b=(3,-1)$, is dealt with by applying the same symmetry argument as in Case 1 above.

\begin{table}[]
    \begin{tabular}{c||c|c|c|c}
     $(x,x-3)$ & $(0,-3)$ & $(1,-2)$ & $(2,-1)$ & $(3,0)$ \\
    \hline
     $\a$ & $(-1,0)$ & $(0,1)$ & $(1,2)$ & $(2,3)$ \\
    \hline
     $\a\cdot\b$ & $-1$ & $3$ & $7$ & $11$
    \end{tabular}
    \vspace{10pt}
    \caption{Computations of $\a$ and $\a\cdot\b$ under the assumption that $\b=(1,-3)$ and $\a+\b$ is of the form $(x,x-3)$.}
    \label{tab:-4}
\end{table}

\underline{Case 3: $b_2-b_1=2$.}
We argue in the same way as the previous two cases.
The class $\a+\b$ is now of the form $(x,x+3)$, and therefore $x\in\set{0,-1,-2,-3}$.
By the usual symmetry argument, assume that $\b=(-3,-1)$. The computations of $\a$ and $\a\cdot\b$ for each possible value of $x$ are summarised in Table \ref{tab:+2}.
All the computed values of $\a\cdot\b$ are obstructed by assumption \eqref{it:A4d}.

\begin{table}[]
    \begin{tabular}{c||c|c|c|c}
     $(x,x-3)$ & $(0,3)$ & $(-1,2)$ & $(-2,1)$ & $(-3,0)$ \\
    \hline
     $\a$ & $(3,4)$ & $(2,3)$ & $(1,2)$ & $(0,1)$ \\
    \hline
     $\a\cdot\b$ & $-5$ & $-3$ & $-1$ & $1$
    \end{tabular}
    \vspace{10pt}
    \caption{Computations of $\a$ and $\a\cdot\b$ under the assumption that $\b=(-3,-1)$ and $\a+\b$ is of the form $(x,x+3)$.}
    \label{tab:+2}
\end{table}

Thus, the only possibility left is that $b_2-b_1=-2$, which leaves out the two cases $\b = (3,1)$ or $(-1,-3)$ that appear in the statement of the lemma.
\end{proof}

The last part of this section is devoted to obstructing the last two remaining cases after Lemma \ref{lem:reduction}, namely
\begin{itemize}
    \item $\a=(a,a+1)$ and $\b = (3,1)$;
    \item $\a=(a,a+1)$ and $\b = (-1,-3)$.
\end{itemize}
To do so, we will add some assumptions on the 3- and 5-signatures of the knots.
We will also need to add more assumptions on the linking number, to ensure that we get enough divisibility to apply Theorem \ref{thm:signature}.

The following formula for the signature function of the positive torus knot $T_{2,q}$ (i.e.\ $q>0$) is a special case of \cite[Proposition 1]{L:signatures}:
\begin{equation}
\label{eq:signaturesT2q}
\sigma_{T_{2,q}}\left(e^{2\pi ix}\right) = 2 \left\lfloor \frac{1}{2} - q\cdot |x| \right\rfloor \qquad \mbox{for } x\in\left[-\frac12,\frac12\right] \setminus \left(\frac1q\Z+\frac1{2q}\right),
\end{equation}
where $\left\lfloor x \right\rfloor$ denotes the floor of $x$.

\begin{remark}
    Equation \eqref{eq:signaturesT2q} determines the values of $\sigma_{T_{2,q}}(e^{2\pi ix})$ also at the jump points, i.e.\ when $x\in\frac1q\Z+\frac1{2q}$, using the property that for every knot $K$ in $S^3$ and every $x \in \R$
    \[
    \sigma_{K}(e^{2\pi ix}) =
    \frac12 \cdot \left(
    \lim_{y \to x^-} \sigma_{K}(e^{2\pi iy}) +
    \lim_{y \to x^+} \sigma_{K}(e^{2\pi iy})
    \right).
    \]
\end{remark}

\begin{remark}
\label{rem:q<0}
When $q<0$, the value of the function $\sigma_{T_{2,q}}(\cdot)$ can be recovered from Equation \eqref{eq:signaturesT2q} using the identity
$\sigma_{T_{2,q}}(e^{2\pi ix}) = -\sigma_{T_{2,-q}}(e^{2\pi ix})$.
\end{remark}

\subsubsection{The 3-signatures}
We make the following two assumptions:
\begin{enumerate}[font=\textbf]
    \item[({\crtcrossreflabel{A4e}[it:A4e]})] \label{test} $\lk(A,B) \equiv 1 \pmod 3$.
    \item[({\crtcrossreflabel{A7}[it:A7]})] $\sigma_A(\zeta_3) = \sigma_B(\zeta_3) \neq +2$.
\end{enumerate}

\begin{lemma}
\label{lem:3-signatures}
Suppose that $L=A\cup B$ is a link satisfying \eqref{it:A1}, \eqref{it:A4a}, \eqref{it:A4e}, and \eqref{it:A7} which is smoothly slice in $\CP2 \# \bCP2$, with discs in homology classes $\a$ and $\b$.

Then $(\a, \b) \neq ((a,a+1),(3,1))$.
\end{lemma}

\begin{proof}
Using \eqref{it:A4a} and \eqref{it:A4e}, we write $\lk(A,B) = 6j+1$ for some $j \in \Z$. Then, using Equation \eqref{eq:lkformula} we compute
\[
\lk(A,B) = -\a\cdot\b = -2a+1,
\]
from which we deduce $a=-3j$.

By Lemma \ref{lem:sumsandcablings} (with $\b^2=8$ and $n=-6j-1$), we have that the knot
\[
K := A \# B_{(2, 12j-14\pm1)}
\]
bounds a smooth disc $D$ in homology class
\[
\a+2\b=(-3j+6,-3j+3),
\]
which is a 3-divisible class. We can therefore apply Theorem \ref{thm:signature} to $K$ and $D$ with $m=3$ and $r=1$, and obtain
\begin{equation}
\label{eq:case31}
    \left| 
    \sigma_K(\zeta_3) - \frac49 \cdot (\a+2\b)^2
    \right|
    \leq
    2.
\end{equation}
The computation of $(\a+2\b)^2$ is straightforward:
\[
(\a+2\b)^2 = (-3j+6)^2 - (-3j+3)^2 = 9 \cdot (-2j+3).
\]
As for the computation of $\sigma_K(\zeta_3)$, we use Theorem \ref{thm:satellite} (together with the fact that $\sigma_{B}(\zeta) = \sigma_{B}(\overline\zeta)$):
\[
\sigma_K(\zeta_3) = \sigma_A(\zeta_3) + \sigma_B(\zeta_3) + \sigma_{T_{2, 12j-14\pm1}}(\zeta_3).
\]
The computation of the last summand is done using Equation \eqref{eq:signaturesT2q} and Remark \ref{rem:q<0}:
\[
\sigma_{T_{2, 12j-14\pm1}}(\zeta_3) = -8j+9\mp1.
\]
Substituting back into Equation \eqref{eq:case31} we obtain 
\begin{equation*}
    \left| 
    \sigma_A(\zeta_3) + \sigma_B(\zeta_3) + (-8j+9\mp1) - \frac49 \cdot 9 (-2j+3)
    \right|
    \leq
    2.
\end{equation*}
This can be simplified as
\[
    \left| 
    \sigma_A(\zeta_3) + \sigma_B(\zeta_3) -3 \mp1
    \right|
    \leq
    2,
\]
which is impossible under assumption \eqref{it:A7}.
\end{proof}

\subsubsection{The 5-signatures}
We make the following two assumptions:
\begin{enumerate}[font=\textbf]
    \item[({\crtcrossreflabel{A4f}[it:A4f]})] $\lk(A,B) \equiv 1 \pmod 5$.
    \item[({\crtcrossreflabel{A8}[it:A8]})] $\sigma_A(\zeta_5) + \sigma_A(\zeta_5^2) = \sigma_B(\zeta_5) + \sigma_B(\zeta_5^2) \geq +2$.
\end{enumerate}

\begin{lemma}
\label{lem:5-signatures}
Suppose that $L=A\cup B$ is a link satisfying \eqref{it:A1}, \eqref{it:A3}, \eqref{it:A4a}, \eqref{it:A4f}, and \eqref{it:A8} which is smoothly slice in $\CP2 \# \bCP2$, with discs in homology classes $\a$ and $\b$.

Then $(\a, \b) \neq ((a,a+1),(-1,-3))$.
\end{lemma}

\begin{proof}
The argument is similar to that of Lemma \ref{lem:3-signatures}.
Using \eqref{it:A4a} and \eqref{it:A4f}, we write $\lk(A,B) = 10k+1$ for some $k \in \Z$. Then, using Equation \eqref{eq:lkformula} we compute
\[
\lk(A,B) = -\a\cdot\b = -2a-3,
\]
from which we deduce $a=-5k-2$.

We apply Lemma \ref{lem:sumsandcablings} with the roles of $A$ and $B$ swapped, and (using $\a^2=10k+3$ and $n=-10k-1$) we have that the knot
\[
K := A_{(2, -4\pm1)} \# B
\]
bounds a smooth disc $D$ in homology class
\[
2\a+\b=(-10k-5, -10k-5),
\]
which is a 5-divisible class with $(2\a+\b)^2=0$. We can therefore apply Theorem \ref{thm:signature} to $K$ and $D$ with $m=5$ and $r=1$, and obtain
\begin{equation}
\label{eq:case-1-3}
    \left| 
    \sigma_K(\zeta_5)
    \right|
    \leq
    2.
\end{equation}
By Theorem \ref{thm:satellite} we get
\[
\sigma_K(\zeta_5) = \sigma_A(\zeta_5^2) + \sigma_B(\zeta_5) + \sigma_{T_{2,-4\pm1}}(\zeta_5).
\]
The last summand equals $+2$ by Equation \eqref{eq:signaturesT2q}.
Substituting back into Equation \eqref{eq:case-1-3}, and using \eqref{it:A3} to replace $\sigma_B(\cdot)$ with $\sigma_A(\cdot)$, we obtain 
\begin{equation*}
    \left| 
    \sigma_A(\zeta_5) + \sigma_A(\zeta_5^2) + 2
    \right|
    \leq
    2,
\end{equation*}
which is impossible under assumption \eqref{it:A8}.
\end{proof}

\subsection{Proof of Theorem \ref{thm:CP2bCP2}}

We can now prove Theorem \ref{thm:CP2bCP2} in the following, more general form.

\begin{theorem}
\label{thm:CP2bCP2assumptions}
Let $L$ be a 2-component link in $S^3$ satisfying assumptions \eqref{it:A1}-\eqref{it:A8}.
Then $L$ is not smoothly slice in $\CP2\#\bCP2$.
\end{theorem} 
\begin{proof}
Let $L = A \cup B$, and suppose by contradiction that it bounds two disjoint smooth discs in homology classes $\a$ and $\b$, respectively.
Then, by Lemma \ref{lem:B1-5} we have two possibilities:
\begin{enumerate}
    \item \label{case:CP-1} either both $\a$ and $\b$ belong to one of the infinite families of Lemma \ref{lem:B1-5}, or
    \item \label{case:CP-2}$\a$ belongs to one of the infinite families and $\b$ is one of the eight sporadic cases $(\pm1, \pm3)$ and $(\pm3, \pm1)$.
\end{enumerate}

To rule out possibility \eqref{case:CP-1}, recall that the symmetries of $\CP2\#\bCP2$, which are spanned by \eqref{it:S1} and \eqref{it:S2}, act transitively on the four infinite families of Lemma \ref{lem:B1-5}, thus we can assume $\a=(a,a+1)$. However, Lemmas \ref{lem:table1}, \ref{lem:table3}, and \ref{lem:table2} rule out this case, depending on what infinite family $\b$ belongs to.

Lastly, to rule out possibility \eqref{case:CP-2}, by Lemma \ref{lem:reduction} we can assume that $\a=(a,a+1)$ and $\b=(3,1)$ or $(-1,-3)$.
These two cases are obstructed by Lemmas \ref{lem:3-signatures} and \ref{lem:5-signatures} respectively.

Thus, $L$ could not be smoothly slice in $\CP2\#\bCP2$.
\end{proof}

In order to find a concrete example of a 2-component link that is not slice in $\CP2\#\bCP2$ we just need to produce a link satisfying all assumptions \eqref{it:A1}-\eqref{it:A8}. Our assumption \eqref{it:A3} on the symmetry of the link implies that it is enough to find a \emph{knot} $K$ with certain properties, and then set $A=B=K$ as knots in $S^3$. We can choose a knot $K$ which satisfies the following conditions:
\begin{itemize}
    \item $g_4(K)=1$;
    \item $\Arf K = 1$;
    \item $\sigma_K(\zeta_2) = \sigma_K(\zeta_5^2) = +2$ and $\sigma_K(\zeta_3) = \sigma_K(\zeta_5) = 0$.
\end{itemize}
A search with KnotInfo \cite{knotinfo} showed that we can choose $K=10_{125}$.

Thus, since the link in Figure \ref{fig:S2xS2structure} satisfies all assumptions \eqref{it:A1}-\eqref{it:A8}, Theorem \ref{thm:CP2bCP2assumptions} implies Theorem \ref{thm:CP2bCP2}.

\section{The search for exotic $\CP2\#\bCP2$'s}

Using a link $L = A \sqcup B$ which is not slice in $\CP2 \# \bCP2$, produced from Theorem \ref{thm:CP2bCP2assumptions}, we can potentially give examples of an exotic $\CP2 \# \bCP2$, using the proposition below.

\begin{proposition}
\label{prop:exotic}
Suppose that $L = A \sqcup B$ is a link in $S^3$ which is not slice in $\CP2 \# \bCP2$, and let $f_A$ and $f_B$ be framings $f_A$ and $f_B$ for the two link components such that at least one of $f_A$ and $f_B$ is odd and the matrix
\[
Q :=
\left(
\begin{matrix}
    f_A & \lk \\
    \lk & f_B
\end{matrix}
\right)
\]
has determinant $-1$, where $\lk = \lk(A,B)$. If the integer homology sphere $Y = S^3_{f_A,f_B}(L)$ bounds an integer homology ball $W$ with $\pi_1(W)$ normally generated by $\pi(Y)$, then there exists an exotic $\CP2 \# \bCP2$.
\end{proposition}

Note that for example every \emph{ribbon} integer homology ball (i.e.\ one that has a handle decomposition with no 3-handles) has $\pi_1$ normally generated by the boundary.

\begin{proof}
The condition that $\det Q = -1$ implies that the 4-manifold $X_{f_A,f_B}(L)$, i.e.\ the trace of $(f_A, f_B)$-surgery on $L$, has $\sigma$ and $b_2$ equal to those of $\CP2 \# \bCP2$, and its boundary is an integer homology sphere $Y$.

If $Y$ bounds an integer homology ball $W$ with $\pi_1(W)$ normally generated by $\pi(Y)$, then the result of the gluing $X:= X_{f_A, f_B}(L) \cup -W$ is simply-connected, by Van Kampen's theorem. The parity assumption on $f_A$ or $f_B$ implies that the resulting manifold is non-spin.
Thus, $X$ is homeomorphic to $\CP2 \# \bCP2$ by Freedman's classification \cite{F:classification}. On the other hand, $X$ cannot be diffeomorphic to $\CP2 \# \bCP2$, because $L$ is obviously smoothly slice in $X$ (by construction), but $L$ is not smoothly slice in $\CP2 \# \bCP2$ by assumption. 
\end{proof}

\begin{remark}
\label{rem:ZHS}
If $L$ is a link as given by Theorem \ref{thm:CP2bCP2assumptions}, then the choice of framings $f_A = \lk^2-1$ and $f_B = 1$ will always produce a matrix $Q$ satisfying the hypotheses of Proposition \ref{prop:exotic}. Of course there is no guarantee that the resulting integer homology sphere $Y$ will bound an integer homology ball $W$ (and that $\pi_1(W)$ is normally generated by $\pi_1(Y)$).

In fact, if we write $\lk = 30 \ell + 1$ (as in assumption \eqref{it:A4}), we can compute that Rokhlin's obstruction $\mu(Y)$ does not vanish if $\ell \equiv 0,1 \pmod 4$, so in these cases $Y$ cannot bound an integer homology ball.
On the other hand, if $\ell\equiv 2,3 \pmod 4$, then $\mu(Y) = 0$, so Rokhlin does not obstruct $Y$ from boundaing an integer homology ball.

To compute $\mu(Y)$ we use \cite[Theorem 2]{Klug}, with $F$ being the closed surface obtained by capping off a Seifert surface for $A$ with the core of the 2-handle attached along $A$; then $\Arf F = \Arf A = 1$ (by assumption \eqref{it:A5}) and $\Arf \partial F = 0$, since $F$ is closed.
\end{remark}

We can now prove Proposition \ref{prop:intro-exotic} from the introduction.

\begin{proof}[Proof of Proposition \ref{prop:intro-exotic}]
We can choose links $L_m$ as in Figure \ref{fig:CP2bCP2link}, but with the $+29$ in the box replaced with a $30\cdot(4m+1)-1$. Note that $L_0$ is precisely the link in Figure \ref{fig:CP2bCP2link}.

In light of Remark \ref{rem:ZHS}, if we choose $f_A = (30\cdot(4m+1)-1)^2-1$ and $f_B=1$, the integer homology sphere $Y_m$ resulting from surgery has vanishing Rokhlin invariant. By Proposition \ref{prop:exotic}, if $Y_m$ bounds an integer homology ball with $\pi_1$ generated by the boundary, we have an exotic $\CP2 \# \bCP2$.
\end{proof}

\begin{remark}
For the specific case when $\lk(A,B) = -29$, we also have that the following choices of framings yield a matrix $Q$ satisfying the hypotheses of Proposition \ref{prop:exotic}:
\begin{itemize}
    \item $f_A = 24$ and $f_B = 35$;
    \item $f_A = 40$ and $f_B = 21$;
    \item $f_A = 120$ and $f_B = 7$.
\end{itemize}
In all three cases the resulting integer homology sphere $Y$ has vanishing Rokhlin invariant.
\end{remark}

\begin{remark}
If we relax the determinant hypothesis in Proposition \ref{prop:exotic} from ``determinant $-1$'' to ``negative determinant'' (so that the resulting surgery trace is indefinite), then we get a similar example of \textit{rational} homology spheres $Y$, where in this case the existence of a \textit{rational} homology ball with $\pi_1$ normally generated by the boundary would imply the existence of an exotic closed manifold.
\end{remark}

\begin{remark}
Our argument can be adapted to produce analogues of Propositions \ref{prop:exotic} and \ref{prop:intro-exotic} for potentially exotic $S^2 \times S^2$ as opposed to $\CP2\#\bCP2$.
This involves finding a family of 2-component links that are not slice in (the standard) $S^2 \times S^2$, which can be done using a similar case analysis to the one in Section \ref{sec:CP2+bCP2}.
We will address this case in a separate note.
\end{remark}

\section{A 3-component link not topologically slice in \texorpdfstring{$\CP2\#\bCP2$}{a blown-up CP2}}

While the methods outlined in Sections \ref{sec:S2xS2} and \ref{sec:CP2+bCP2} work only in the smooth category, we can improve the result of Theorem \ref{thm:Yasuhara-null-homologous-links} in the case of $\CP2 \# \bCP2$.

\begin{theorem}
    \label{thm:3-cpt-link-not-top-slice-in-CP2+bCP2}
    Let $L = H \sqcup C$ be a 3-component link in $S^3$ given by the split union of a Hopf link $H=A \cup B$ and a knot $C$ satisfying the following properties:
\begin{itemize}
    \item $C$ is topologically slice in neither $\CP2$ nor $\bCP2$;
    \item $C$ is not topologically H-slice in $\CP2 \# \bCP2$.
\end{itemize}
    Then $L$ is not topologically slice in $\CP2\#\bCP2$.
\end{theorem}

The existence of a knot $C$ that is topologically slice in neither $\CP2$ nor $\bCP2$ follows by a straightforward variation of an argument of Kasprowski-Powell-Ray-Teichner (cf.\ \cite[Corollary 1.15.(2)]{KPTR:surfaces}), which we briefly outline below.

\begin{proposition}
\label{prop:KPRT-powerup}
The knot $C = \#^7 T_{2,3}$ is topologically slice in neither $\CP2$ nor $\bCP2$.
\end{proposition}
Note that $\#^7 T_{2,3}$ is also not topologically H-slice in $\CP2 \# \bCP2$, by Theorem \ref{thm:signature} applied with $[\Sigma] = 0$, $m=2$, and $r=1$. Thus, it satisfies both conditions of Theorem \ref{thm:3-cpt-link-not-top-slice-in-CP2+bCP2}. 
\begin{proof}
We identify $H_2(\CP2;\Z) \cong H_2(\bCP2;\Z) \cong \Z$. For every homology class $d\in\Z$, we rule out the possibility that $C$ bounds a locally flat disc in $\CP2$ or $\bCP2$ in that class.

Since $\Arf C = 1$, the same argument as in \cite[Corollary 1.15.(2)]{KPTR:surfaces} shows that $C$ does not bounds a disc with $d=\pm1$.

For every other $d$, the corresponding homology class is not primitive, so we can apply Theorem \ref{thm:signature}. We prepare for it by defining, for each prime power $m$ and a knot $K$, the `central' signature
\[
\sigma^{\mathrm{centr}}_m(K) :=
\begin{cases}
\sigma_K(-1) & \mbox{if $m$ is even}\\
\sigma_K\left(e^{\pi i \cdot \frac{m-1}{m}}\right) & \mbox{if $m$ is odd}
\end{cases}
\]
In our case, when $K=C$, we have $\sigma^{\mathrm{centr}}_m(K) = -14$ for all prime powers $m$.

Suppose by contradiction that $C$ bounds a disc in $\CP2$ or $\bCP2$ in homology class $d\neq \pm1$. If $m$ is a prime power that divides $d$, then using Theorem \ref{thm:signature}, the triangle inequality $|x\pm y| \geq \left||x|-|y|\right|$, and $\sigma^{\mathrm{centr}}_m(K) = -14$, we obtain
\begin{equation}
\label{eq:KPRT-powerup}
1 \geq \left|14-|f_m(d)|\right|,
\end{equation}
where
\[
f_m(d) :=
\begin{cases}
\frac{d^2}2-1 & \mbox{if $m$ is even}\\
\frac{d^2}2\cdot \frac{m^2-1}{m^2} -1 & \mbox{if $m$ is odd}
\end{cases}
\]

If $|d|\leq 5$, $d\neq1$, then it is immediate to check that $|f_m(d)|\leq 11$ for every prime power factor $m$ of $d$, so Equation \eqref{eq:KPRT-powerup} is not satisfied.

If $d = \pm6$, then we can choose $m=2$, and since $f_2(6) = 17$, again Equation \eqref{eq:KPRT-powerup} is not satisfied.

Finally, if $|d| \geq 7$, we have
\[
f_m(d) \geq \frac{d^2}2\cdot \frac{8}{9} - 1 > 20,
\]
and once again Equation \eqref{eq:KPRT-powerup} is not satisfied.

Thus, there cannot be any value of $d$ such that $C$ bounds a disc in $\CP2$ or $\bCP2$ in homology class $d$.
\end{proof}

%
%

\begin{proof}[Proof of Theorem {\ref{thm:3-cpt-link-not-top-slice-in-CP2+bCP2}}]
Suppose by contradiction that there exists three disjoint discs $D_A$, $D_B$, and $D_C$ in $\CP2 \# \bCP2$ with boundary $A$, $B$, and $C$ respectively.

The fact that $C$ is not topologically H-slice in $\CP2 \# \bCP2$ shows that $[D_C] \neq 0$. By Equation \eqref{eq:lkformula}, $[D_A] \cdot [D_C] = [D_B] \cdot [D_C] = 0$, i.e., the homology classes $[D_A]$ and $[D_B]$ are orthogonal to $[D_C]$, and since the intersection pairing on $\CP2 \# \bCP2$ is non-degenerate and of rank 2, $[D_A]$ and $[D_B]$ must in fact be linear multiples of each other. From Equation \eqref{eq:lkformula} we get
\[
[D_A] \cdot [D_B] = \pm1,
\]
and therefore $[D_A]$ and $[D_B]$ are primitive. Thus, we must have $[D_A]=\pm[D_B]$, and the previous equation implies that $D_A$ is a $(\pm1)$-framed disc with boundary $A$.

If we remove a neighbourhood of $D_A$ (which is a punctured $\CP2$ or $\bCP2$), then the complement $X$ is again a punctured $\CP2$ or $\bCP2$, by Freedman's classification theorem \cite{F:classification}.


Thus, we obtain a contradiction because by construction $C$ bounds a disc $D_C$ in $X$, but by hypothesis $C$ is topologically slice in neither $\CP2$ nor $\bCP2$.
\end{proof}

\appendix

\section*{Appendix: Assumptions for the link not smoothly slice in \texorpdfstring{$\CP2 \# \bCP2$}{a blown up CP2}}
\label{appendix:CP2bCP2}
\begin{enumerate}
    \item[\eqref{it:A1}] $L$ has a diagram as in Figure \ref{fig:S2xS2structure}, where $T_A$ (resp.\ $T_B$) is a $(1,1)$-tangle whose closure is $A$ (resp.\ $B$), and $n \in \Z$ is the number of right-handed full twists added in the region.
    \item[\eqref{it:A2}] $g_{B^4}(A) = g_{B^4}(B) = 1$.
    \item[\eqref{it:A3}] The link $L$ in Figure \ref{fig:S2xS2structure} has an ambient isotopy that swaps $A$ and $B$.
    \item[({\crtcrossreflabel{A4}[it:A4]})] $\lk(A,B) = 30\ell+1$ for some $\ell\neq0$.
    \item[\eqref{it:A5}] $\Arf A = \Arf B = 1$.
    \item[\eqref{it:A6}] $\sigma_A(\zeta_2) = \sigma_B(\zeta_2) \neq 0$.
    \item[\eqref{it:A7}] $\sigma_A(\zeta_3) = \sigma_B(\zeta_3) \neq +2$.
    \item[\eqref{it:A8}] $\sigma_A(\zeta_5) + \sigma_A(\zeta_5^2) = \sigma_B(\zeta_5) + \sigma_B(\zeta_5^2) \geq +2$.
\end{enumerate}

Assumption \eqref{it:A4} is made to subsume all the following assumptions on the linking number:
\begin{enumerate}
    \item[\eqref{it:A4a}] $\lk(A,B) \not\equiv 0 \pmod 2$.
    \item[\eqref{it:A4b}] $\lk(A,B) \not\equiv 0 \pmod 3$.
    \item[\eqref{it:A4c}] $|\lk(A,B)| \geq 4$.
    \item[\eqref{it:A4d}] $\lk(A,B) \notin \set{\pm1, \pm3, \pm5, \pm7, \pm11}$.
    \item[\eqref{it:A4e}] $\lk(A,B) \equiv 1 \pmod 3$.
    \item[\eqref{it:A4f}] $\lk(A,B) \equiv 1 \pmod 5$.
\end{enumerate}

\bibliographystyle{alpha}
\bibliography{bibliography}

\newcommand{\etalchar}[1]{$^{#1}$}
\begin{thebibliography}{MMSW23}

\bibitem[ACM{\etalchar{+}}23]{ACMPS}
Paolo Aceto, Nickolas~A Castro, Maggie Miller, JungHwan Park, and Andr{\'a}s
  Stipsicz.
\newblock Slice obstructions from genus bounds in definite 4-manifolds.
\newblock {\em arXiv preprint arXiv:2303.10587}, 2023.

\bibitem[Akb91]{Akbulut}
Selman Akbulut.
\newblock A fake compact contractible 4-manifold.
\newblock {\em Journal of Differential Geometry}, 33(2):335--356, 1991.

\bibitem[AP10]{AP}
Anar Akhmedov and B~Doug Park.
\newblock Exotic smooth structures on small 4-manifolds with odd signatures.
\newblock {\em Inventiones mathematicae}, 181(3):577--603, 2010.

\bibitem[DKM{\etalchar{+}}22]{DKMPS}
Irving Dai, Sungkyung Kang, Abhishek Mallick, JungHwan Park, and Matthew
  Stoffregen.
\newblock The $(2, 1)$-cable of the figure-eight knot is not smoothly slice.
\newblock {\em arXiv preprint arXiv:2207.14187}, 2022.

\bibitem[FGMW10]{FGMW}
Michael~H Freedman, Robert~E Gompf, Scott Morrison, and Kevin Walker.
\newblock Man and machine thinking about the smooth 4-dimensional
  {P}oincar{\'e} conjecture.
\newblock {\em Quantum Topology}, 1(2):171--208, 2010.

\bibitem[FK78]{FK:Rochlin}
Michael Freedman and Robion Kirby.
\newblock A geometric proof of {R}ochlin's theorem.
\newblock In {\em Algebraic and geometric topology ({P}roc. {S}ympos. {P}ure
  {M}ath., {S}tanford {U}niv., {S}tanford, {C}alif., 1976), {P}art 2}, volume
  XXXII of {\em Proc. Sympos. Pure Math.}, pages 85--97. Amer. Math. Soc.,
  Providence, RI, 1978.

\bibitem[Fre82]{F:classification}
Michael~Hartley Freedman.
\newblock The topology of four-dimensional manifolds.
\newblock {\em Journal of Differential Geometry}, 17(3):357--453, 1982.

\bibitem[Gil81]{G:G-signature}
Patrick~M Gilmer.
\newblock Configurations of surfaces in 4-manifolds.
\newblock {\em Transactions of the American Mathematical Society},
  264(2):353--380, 1981.

\bibitem[GS23]{GS}
Robert~E Gompf and Andr{\'a}s~I Stipsicz.
\newblock {\em 4-manifolds and {K}irby calculus}, volume~20.
\newblock American Mathematical Society, 2023.

\bibitem[Kin87]{Kinoshita}
Shin'ichi Kinoshita.
\newblock On $\theta_n$-curves on $\mathbb{R}^3$ and their constituent knots.
\newblock In Shin'ichi Suzuki, editor, {\em Topology and Computer Science
  (Proceedings of the Symposium held in honour of Shin’ichi Kinoshita,
  Hiroshi Noguchi and Tatsuo Homma on the occasion of their sixtieth
  birthday)}, pages 211--216. Kinokuniya Company Ltd., Tokyo, 1987.

\bibitem[Kir06]{K:topology}
Robion~C Kirby.
\newblock {\em The topology of 4-manifolds}, volume 1374.
\newblock Springer, 2006.

\bibitem[Klu20]{Klug}
Michael Klug.
\newblock A relative version of {R}ochlin's theorem.
\newblock {\em arXiv preprint arXiv:2011.12418}, 2020.

\bibitem[KPRT22]{KPTR:surfaces}
Daniel Kasprowski, Mark Powell, Arunima Ray, and Peter Teichner.
\newblock Embedding surfaces in 4-manifolds.
\newblock {\em arXiv preprint arXiv:2201.03961}, 2022.

\bibitem[KR21]{KR}
Michael Klug and Benjamin Ruppik.
\newblock Deep and shallow slice knots in 4-manifolds.
\newblock {\em Proceedings of the American Mathematical Society, Series B},
  8(17):204--218, 2021.

\bibitem[Lit06]{L:signatures}
Richard~A Litherland.
\newblock Signatures of iterated torus knots.
\newblock In {\em Topology of Low-Dimensional Manifolds: Proceedings of the
  Second Sussex Conference, 1977}, pages 71--84. Springer, 2006.

\bibitem[LM23]{knotinfo}
Charles Livingston and Allison~H. Moore.
\newblock Knotinfo: Table of knot invariants.
\newblock URL: \url{knotinfo.math.indiana.edu}, August 2023.

\bibitem[LW97]{LW:locallyflatspheres}
Ronnie Lee and Dariusz~M Wilczy{\'n}ski.
\newblock Representing homology classes by locally flat surfaces of minimum
  genus.
\newblock {\em American Journal of Mathematics}, 119(5):1119--1137, 1997.

\bibitem[MM22]{MM:slice}
Marco Marengon and Stefan Mihajlovi{\'c}.
\newblock Unknotting number 21 knots are slice in {K3}.
\newblock {\em arXiv preprint arXiv:2210.10089}, 2022.

\bibitem[MMP24]{MMP}
Ciprian Manolescu, Marco Marengon, and Lisa Piccirillo.
\newblock Relative genus bounds in indefinite four-manifolds.
\newblock {\em Mathematische Annalen}, pages 1--26, 2024.

\bibitem[MMSW23]{MMSW}
Ciprian Manolescu, Marco Marengon, Sucharit Sarkar, and Michael Willis.
\newblock A generalization of {R}asmussen’s invariant, with applications to
  surfaces in some four-manifolds.
\newblock {\em Duke Mathematical Journal}, 172(2):231--311, 2023.

\bibitem[MY97]{MY:generalized}
Katura Miyazaki and Akira Yasuhara.
\newblock Generalized $\#$-unknotting operations.
\newblock {\em Journal of the Mathematical Society of Japan}, 49(1):107--123,
  1997.

\bibitem[Nor69]{N:slice}
RA~Norman.
\newblock Dehn's lemma for certain 4-manifolds.
\newblock {\em Inventiones mathematicae}, 7:143--147, 1969.

\bibitem[Ren23a]{Ren1}
Qiuyu Ren.
\newblock Lee filtration structure of torus links.
\newblock {\em arXiv preprint arXiv:2305.16089}, 2023.

\bibitem[Ren23b]{Ren2}
Qiuyu Ren.
\newblock Slice genus bound in {$DTS^2$} from $s$-invariant.
\newblock {\em arXiv preprint arXiv:2306.17816}, 2023.

\bibitem[Rub96]{R:minimalgenus}
Daniel Ruberman.
\newblock The minimal genus of an embedded surface of non-negative square in a
  rational surface.
\newblock {\em Turkish Journal of Mathematics}, 20(1):129--133, 1996.

\bibitem[Sco22]{S:wild}
Alexandru Scorpan.
\newblock {\em The wild world of 4-manifolds}.
\newblock American Mathematical Society, 2022.

\bibitem[Suz69]{S:slice}
Shin'ichi Suzuki.
\newblock Local knots of 2-spheres in 4-manifolds.
\newblock {\em Proceedings of the Japan Academy}, 45(1):34--38, 1969.

\bibitem[Vir75]{V:G-signature}
Oleg~Y Viro.
\newblock Placements in codimension 2 and boundary.
\newblock {\em Uspekhi Mat. Nauk}, 30(1):231--232, 1975.

\bibitem[Yas96a]{Y:CL}
Akira Yasuhara.
\newblock Connecting lemmas and representing homology classes of simply
  connected 4-manifolds.
\newblock {\em Tokyo Journal of Mathematics}, 19(1):245--261, 1996.

\bibitem[Yas96b]{Y:LinkHomology}
Akira Yasuhara.
\newblock Link homology in 4-manifolds.
\newblock {\em Bulletin of the London Mathematical Society}, 28(4):409--412,
  1996.

\end{thebibliography}

\end{document}